\documentclass{amsart}

\usepackage[latin1]{inputenc}
\usepackage[T1]{fontenc}
\usepackage[french]{babel}
\usepackage{amssymb,amsmath,mathrsfs}\usepackage[all]{xy}

\usepackage{enumitem}
\usepackage{xcolor}

\DeclareMathOperator{\Ker}{Ker}

\newcommand{\ZZ}{{\mathbb{Z}}}

\newtheorem{theorem}{Théorème}[section]

\newtheorem{lemma}[theorem]{Lemme}
\newtheorem{conjecture}[theorem]{Conjecture}
\newtheorem{proposition}[theorem]{Proposition}
\newtheorem{definition}[theorem]{Définition}
\newtheorem{corollary}[theorem]{Corollaire}

\begin{document}

\title{Homothéties explicites des représentations $\ell$-adiques}

\author{Aur\'elien Galateau \& C\'esar Mart\'inez}
\address{Universit\'e de Cergy-Pontoise}
\email{aurelien.galateau@u-cergy.fr}
\address{Universität Regensburg}
\email{cesar.martinez@ur.de}

\begin{abstract}
Cet article présente et précise les principaux résultats connus sur la taille du sous-groupe des homothéties des représentations $\ell$-adiques associées à la torsion d'une variété abélienne. De telles estimations permettent notamment de donner des bornes uniformes explicites dans le cadre du problème de Manin-Mumford.
\end{abstract}

\maketitle

\section{Introduction}

La question de la distribution des points de torsion d'une courbe algébrique plongée dans sa jacobienne est posée par Lang en 1965 (\cite{Lang65}). Inspiré par des travaux de Manin et par ses discussions avec Mumford, il formule une conjecture qui sera démontrée une vingtaine d'années plus tard.
\begin{theorem}[Raynaud]
Si $C$ est une courbe algébrique de genre au moins deux, définie sur un corps de nombres et plongée dans sa jacobienne, elle contient un nombre fini de points de torsion. 
\end{theorem}

À la lumière des travaux d'Ihara, Serre et Tate sur l'analogue torique de ce problème, Lang montre que certaines propriétés galoisiennes des points de torsion pourraient permettre de résoudre la question posée par Manin et Mumford. 

\medskip

Soit $A$ une variété abélienne de dimension $g \geq 1$ définie sur un corps de nombres $K$ dont on fixe une clôture algébrique $\bar{K}$. Si $\ell$ est dans l'ensemble $\mathcal{P}$ des nombres premiers, l'action du groupe de Galois $G_K:= \mathrm{Gal}(\bar{K}/K)$ sur le module de Tate $T_{\ell}$ de $A$ (dont on fixe une base sur $\mathbb{Z}_{\ell}$) définit une représentation : $$\rho_{\ell}: G_K \rightarrow \mathrm{GL}_{\mathbb{Z}_{\ell}}(T_{\ell}) \simeq \mathrm{GL}_{2g}(\mathbb{Z}_{\ell}),$$ 
dont l'image sera désormais notée $G_{\ell}$. L'hypothèse faite par Lang est devenue la conjecture suivante, toujours ouverte aujourd'hui.
\begin{conjecture}[Lang]
Si $\ell$ est assez grand, le groupe $G_{\ell}$ contient le sous-groupe des homothéties $\ZZ_{\ell}^{\times}$ de $\mathrm{GL}_{2g}(\mathbb{Z}_{\ell}).$ \end{conjecture}
 
De nombreux résultats sont connus en direction de cette conjecture. Bogomolov a d'abord démontré, en utilisant la théorie de Hodge-Tate, que les représentations $\ell$-adiques contiennent les homothéties à un indice fini près (voir \cite{Bogomolov80}).
\begin{theorem}[Bogomolov]
\label{bogomolov}
Le groupe $G_{\ell}$ contient un sous-groupe ouvert du groupe $\ZZ_{\ell}^{\times}$ des homothéties de $\mathrm{GL}_{2g}(\mathbb{Z}_{\ell}).$ 
\end{theorem}

Si on considère le \og groupe de monodromie \fg \, $H_{\ell}$, c'est-à-dire l'enveloppe algébrique de $G_{\ell}$ dans $\mathrm{GL}_{2g}$ sur $\mathbb{Q}_{\ell}$, les travaux de Bogomolov montrent en fait que $G_{\ell}$ est un sous-groupe ouvert -- pour la topologie $\ell$-adique -- de $H_{\ell}(\mathbb{Q}_{\ell})$. Serre, avec des arguments plus sophistiqués, a ensuite donné une version uniforme du résultat de Bogomolov.
\begin{theorem}[Serre]
\label{homothéties serre}
Il existe un entier $c(A)>0$ tel que : $$\forall \ell \in \mathcal{P}, \; (\ZZ_{\ell}^{\times})^{c(A)} \subset G_{\ell}.$$
\end{theorem}

\noindent
Ce résultat, annoncé dans ses cours au Collège de France (\cite{Serre86}, 136), a ensuite été exposé en détail par Wintenberger (\cite{Wintenberger02}).
\medskip

Le théorème de Serre se révèle être suffisant pour mener à bien la stratégie imaginée par Lang dans le problème de Manin-Mumford. Suivant cette voie, Hindry a généralisé le théorème de Raynaud aux sous-variétés des variétés semi-abéliennes (\cite{Hindry88}). 

\medskip

En combinant une variante de cette approche développée par Beukers et Smyth (\cite{Beukers-Smyth02}) avec de nouveaux arguments d'interpolation, nous avons obtenu dans \cite{Galateau-Martinez17} des bornes uniformes permettant de quantifier la torsion dans les sous-variétés des variétés abéliennes. Pour rendre ces estimations totalement explicites, il faudrait pouvoir estimer la \og constante de Serre \fg, c'est-à-dire l'exposant apparaissant dans le Théorème \ref{homothéties serre}. Nous nous proposons ici de rassembler et de préciser les résultats connus sur cette question.

\subsection*{Plan de l'article}

Nous revisitons d'abord les travaux de Serre et Wintenberger, qui montrent que pour un nombre premier $\ell$ suffisamment grand, le groupe $G_{\ell}$ contient un sous-groupe \og assez grand \fg \, du groupe dérivé $H_{\ell}(\mathbb{Q}_{\ell})'$. Il est alors possible de borner simplement l'exposant du groupe d'homothéties en fonction de $g$. Toute la difficulté ici est d'estimer la taille des nombres premiers pour lesquels ce phénomène a lieu, et ceci a été partiellement réalisé par Zywina (\cite{Zywina19}).

\medskip

Nous reprenons ensuite des résultats classiques de cohomologie galoisienne sur les corps locaux démontrés par Tate dans \cite{Tate67}, que nous précisons en estimant les sauts de ramification d'une extension procyclique. La détermination de certains groupes de cohomologie permet d'exploiter la décomposition de Hodge-Tate des représentations abéliennes $\ell$-adiques pour montrer leur algébricité locale.

\medskip

Le cas des grands premiers étant essentiellement compris, il suffit alors de donner une version explicite du théorème de Bogomolov. Ce problème se décompose en deux parties. La première consiste à contrôler l'écart entre $G_{\ell}$ et $H_{\ell}(\mathbb{Q}_{\ell})'$, ce qui semble actuellement hors d'atteinte. La seconde revient à estimer la taille des homothéties dans une représentation abélienne qui a la propriété d'être de Hodge-Tate. Nous le faisons ici en reprenant l'exposé donné par Serre dans \cite{Serre68}, III A.

\medskip

Nous revenons enfin sur le problème de Manin-Mumford et nous explicitons les résultats uniformes déjà connus dans deux cas particuliers importants : les puissances de courbes elliptiques et les variétés abéliennes CM. De telles bornes, lorsqu'elles sont calculables, permettent de déterminer effectivement le lieu de torsion recherché.

\medskip

Nous tenons à remercier la Deutsche Forschungsgemeinschaft et le programme SFB 1085 : {\it Higher Invariants -- Interactions between Arithmetic Geometry and Global Analysis} pour leur soutien pendant  la conception et la rédaction de cet article.

\section{Constante de Serre pour les grands premiers}

Serre a d'abord exposé la preuve du Théorème \ref{homothéties serre} dans son cours au Collège de France de 1985-1986. Les grandes lignes de la démonstration apparaissent dans sa correspondance de l'époque (en particulier la lettre à Bertrand \cite{Serre86}, 134). Elles ont été reprises par Wintenberger dans \cite{Wintenberger02}, où il est aussi prouvé que la conjecture de Mumford-Tate entraîne la conjecture de Lang. 

\subsection{Extension des scalaires}

Un certain nombre de simplifications apparaissent à condition de se placer sur une extension finie de $K$, dont le degré peut être borné explicitement en fonction de $g$. C'est ce qu'on fera dans la suite de cet article.

On commencera par choisir $K$ assez grand pour que $A$ ait réduction semi-stable en tout idéal premier de l'anneau des entiers $\mathcal{O}_K$. Il suffit pour cela de remplacer $K$ par $K(A[12])$, qui est une extension de $K$ de degré $\leq 12^{4g^2}$ (\cite{Grothendieck72}, Exposé IX, proposition 4.7).
Une propriété intéressante apparaît avec l'extension qu'on a fixée : les groupes de monodromie sont simultanément connexes.

\begin{lemma}
Pour tout $\ell \in \mathcal{P}$, $H_{\ell}$ est connexe.
\end{lemma}

\begin{proof}
Voir \cite{Larsen-Pink97}, théorème 0.1, qui précise le résultat démontré par Serre dans sa lettre à Ribet (\cite{Serre86}, 133). Le plus petit corps de définition rendant connexes tous les $H_{\ell}$ est inclus dans : $$\bigcap_{\ell \in \mathcal{P}}K(A[\ell^{\infty}]) \subset K(A[2^{\infty}]) \cap K(A[3^{\infty}]).$$ Son degré sur $K$ est donc borné par : $$\prod_{k=1}^{2g}(2^k-1)(3^k-1) \leq 6^{3g^2},$$
et on voit que l'extension choisie pour assurer la semi-stabilité partout convient également ici.
\end{proof}

On peut aussi se demander si les représentations $\rho_{\ell}$ sont indépendantes dans le sens suivant.

\begin{definition}
Les $(\rho_{\ell})_{\ell \in \mathcal{P}}$ sont indépendantes si le morphisme produit : $$\rho:= (\rho_{\ell})_{\ell \in \mathcal{P}} : G_K \longrightarrow \prod_{\ell} G_{\ell}$$ est surjectif.
\end{definition}

\noindent
Dans ce cas, on a : $\rho(G_K) = \prod_{\ell} G_{\ell}$ et il est possible de recoller les résultats obtenus sur chaque représentation $\ell$-adique.

\begin{proposition}
Quitte à étendre $K$, on peut considérer que les $\rho_{\ell}$ sont indépendantes.
\end{proposition}

\begin{proof}
Suivant \cite{Serre13}, théorème 1, il suffit de remplacer $K$ par son extension finie : $$\bigcap_{\ell \in \mathcal{P}} K_{\ell}',$$ où $K_{\ell}'$ est le corps de rationalité de tous les points de torsion d'ordre premier à $\ell$. Une version moins précise de ce résultat figure dans la Lettre de Serre à Ribet, \cite{Serre86}, 138.
\end{proof}

\subsection{Groupes dérivés}

On considère maintenant l'enveloppe algébrique $\mathcal{H}_{\ell}$ de $G_{\ell}$ dans le groupe algébrique $\mathrm{GL}_{2g}$ sur $\mathbb{Z}_{\ell}$, dont la fibre générique est $H_{\ell}$.  Son groupe dérivé au sens des schémas en groupes réductifs (\cite{Grothendieck70}, Exposé XXII, 6.2) sera noté $\mathcal{S}_{\ell}$. Le principal résultat en direction du Théorème \ref{homothéties serre} est le suivant (\cite{Wintenberger02}, théorème 2).

\begin{theorem}[Wintenberger] \label{wintenberger}
Il existe un entier $\ell_0$ tel que pour tout premier $\ell \geq \ell_0$ : $$\mathcal{S}_{\ell}(\mathbb{Z}_{\ell})' \subset G_{\ell}.$$
\end{theorem} 

\noindent
Ce résultat repose en grande partie sur une étude de la représentation modulo $\ell$, donnée par l'action de $G_K$ sur la $\ell$-torsion : $$\rho(\ell): G_K \rightarrow \mathrm{GL}_{\mathbb{F}_{\ell}}(A[\ell]) \simeq \mathrm{GL}_{2g}(\mathbb{F}_{\ell}),$$ 
dont l'image sera notée $G(\ell)$. Dans sa lettre à Bertrand (\cite{Serre86}, 134), Serre explique qu'on pourrait trouver une borne effective à condition d'obtenir une version précise d'un théorème de Faltings affirmant la semi-simplicité de cette représentation. Wintenberger (\cite{Wintenberger02}, remarque 2.2) fait un inventaire des obstacles à l'effectivité.

\medskip

Dans un travail récent, Zywina utilise la version donnée par Masser et Wüstholz du théorème de Faltings pour estimer partiellement l'entier $\ell_0$. Par des travaux de Serre, le rang $r$ de $H_{\ell}$ est indépendant de $\ell$ et il existe un idéal premier $\mathfrak{p}$ de $\mathcal{O}_K$ tel que le tore de Frobénius $T_{\mathfrak{p}}$, c'est-à-dire le groupe multiplicatif engendré par les racines du polynôme caractéristique du Frobénius en $\mathfrak{p}$, soit de rang $r$ (\cite{Serre86}, 133 ou \cite{Zywina19}, 2.5). On fixe un tel idéal $\mathfrak{p}$ de norme minimale, et on note $h_F(A)$ la hauteur de Faltings de~$A$. 

\begin{theorem}[Zywina]\label{zywina}
Il existe des réels positifs $\alpha$ et $\beta$ dépendant de $g$ tels qu'on puisse choisir : $$\ell_0=\alpha \cdot \max \{ [K : \mathbb{Q}], h_F(A), N(\mathfrak{p}) \}^{\beta}.$$
\end{theorem}

\noindent
{\it Remarque.} Zywina montre qu'il est possible de faire disparaître la dépendance en l'idéal $\mathfrak{p}$ si on admet l'Hypothèse de Riemann généralisée. En utilisant le théorème de Minkowksi et la comparaison entre \og conducteur \fg \, et hauteur de Faltings (\cite{Pazuki16}, théorème 1.1), la borne devient alors : 
$$\ell_0=\alpha \cdot \max \{ \log \Delta_K, h_F(A) \}^{\beta},$$
où $\Delta_K$ est le discriminant absolu du corps $K$ (\cite{Zywina19}, théorème 1.4).

\subsection{Version effective du théorème de Faltings} Il est possible d'aller un peu plus loin en utilisant les travaux de Gaudron et Rémond sur les degrés d'isogénies, qui donnent des informations précises sur la représentation $\rho(\ell)$ modulo $\ell$. On introduit la notation suivante :
$$\Xi(A):= \Big( (7g)^{8g^2} [K: \mathbb{Q}] \max \big\{ 1, h_F(A), \log [K: \mathbb{Q}] \big\} \Big)^{2g^2},$$
qui apparaît de façon répétée dans \cite{Gaudron-Remond20}. 
\begin{proposition}\label{gaudron-rémond}
Si $\ell > g^{2g^2} \Xi(A)^{g}$, la représentation $\rho(\ell)$ est semi-simple de commutant $\mathbb{F}_{\ell} \otimes \mathrm{End}(A)$. 
\end{proposition}
\begin{proof}
Le théorème 1.9 (4) de \cite{Gaudron-Remond20} montre que le commutant de $\rho(\ell)$ est $\mathbb{F}_{\ell} \otimes \mathrm{End}(A)$ si : $$\ell >  \Xi(A).$$

\noindent
On note $d(A)$ la valeur absolue du discriminant de $\mathrm{End}_K(A)$. Par le théorème du (double) commutant, la semi-simplicité de $\rho(\ell)$ découle de celle de l'algèbre $\mathbb{F}_{\ell} \otimes \mathrm{End}(A)$, qui est réalisée si $\ell$ ne divise pas $d(A)$. On peut borner le discriminant en combinant le lemme 2.10 de \cite{Remond18} et le théorème 1.9 (3) de \cite{Gaudron-Remond20} (le rang de End$_K(A)$ est majoré par $2g^2$) :
\[
d(A) \leq g^{2g^2} \mathrm{vol}(\mathrm{End}_K(A))^2 \leq g^{2g^2}\Xi(A)^g.
\]
Pour : $$\ell > g^{2g^2} \Xi(A)^{g},$$ les deux conditions demandées sur $\rho(\ell)$ sont donc vérifiées.
\end{proof} 





\medskip

On peut alors estimer en partie l'exposant qui apparaît dans le Théorème \ref{zywina}.
\begin{corollary}
Il existe des réels positifs $\alpha, \beta$ dépendant de $g$ tels qu'on puisse choisir : $$\ell_0=\alpha \cdot N(\mathfrak{p})^{\beta} \cdot  \max \{ [K : \mathbb{Q}], h_F(A) \}^{2g^3}.$$
\end{corollary}
\begin{proof}
En suivant la preuve du Théorème \ref{zywina}, on voit que la seule contrainte sur l'exposant $\gamma$ de $\max \{ [K : \mathbb{Q}], h_F(A) \}$ est liée à l'utilisation du théorème de Faltings. La Proposition \ref{gaudron-rémond} montre qu'on peut prendre $\gamma=2g^2 \cdot g=2g^3$.
\end{proof}

\noindent\textit{Remarques.}
Notons que si $A'$ est une variété abélienne $K$-isogène à $A$, les représentations $\ell$-adiques associées à $A$ et $A'$ sont conjuguées dans GL$_{2g}(\mathbb{Q}_{\ell})$. Leurs sous-groupes d'homothéties sont donc identiques. 

\medskip

On peut théoriquement déterminer $\alpha$ et $\beta$ en fonction de $g$. Cela demanderait de rendre effectifs de nombreux résultats profonds de théorie des groupes et de théorie des représentations (dont des théorèmes de Jordan et de Nori), puis de reprendre en détail la preuve du Théorème \ref{zywina}.

\subsection{Homothéties pour les grands premiers}

Lorsque $\ell$ est assez grand, on peut donner une borne assez simple pour la constante de Serre (l'exposant du Théorème \ref{homothéties serre}).

\begin{proposition}
On suppose que $\ell \geq \max \{ \ell_0, \Delta_K + 1\}$. Alors il existe $c \, | (2g)!$ tel que : $$(\mathbb{Z}_{\ell}^{\times})^c \subset G_{\ell}.$$
\end{proposition}

\begin{proof}
On reprend les arguments donnés dans \cite{Wintenberger02}, 2.3. Par une remarque de Deligne (voir \cite{Serre77}, 2, 3), on sait déjà que le groupe des homothéties $\mathbb{Z}_{\ell}^{\times}$ est inclus dans $H_{\ell}(\mathbb{Q}_{\ell})$. Les travaux de Tate aux places de bonne réduction, étendus par Raynaud aux places quelconques (\cite{Bogomolov80}, 2), montrent que la restriction de $\rho_{\ell}$ à tout groupe de décomposition en une place divisant $\ell$ est de Hodge-Tate. Par \cite{Bogomolov80}, 1 (preuve du corollaire), c'est encore vrai pour la représentation abélienne : $$\rho_{\ell}^{\mathrm{ab}} : G_K \longrightarrow H_{\ell}(\mathbb{Q}_{\ell})/H_{\ell}(\mathbb{Q}_{\ell})'.$$ 
Comme $A$ admet réduction semi-stable partout et $K$ est non ramifié au-dessus de $\ell$, on peut appliquer le corollaire 2.4 de \cite{Wintenberger02a} à $\rho_{\ell}^{\mathrm{ab}}$, et on en déduit que $\rho_{\ell}^{\mathrm{ab}} (G_K)$ contient l'image de $\mathbb{Z}_{\ell}^{\times}$ dans $H_{\ell}(\mathbb{Q}_{\ell})/H_{\ell}(\mathbb{Q}_{\ell})'$. Par le Théorème \ref{wintenberger}, si $c$ est un annulateur du quotient $\mathcal{S}_{\ell}(\mathbb{Z}_{\ell}) / \mathcal{S}_{\ell}(\mathbb{Z}_{\ell})'$, on a : $$(\mathbb{Z}_{\ell}^{\times})^c \subset G_{\ell}.$$ Si on note $\pi_{\ell}$ le morphisme de projection associé au revêtement universel de $\mathcal{S}_{\ell}$, on a : $\mathcal{S}_{\ell}(\mathbb{Z}_{\ell})' = \mathrm{Im} \, \pi_{\ell} (\mathbb{Z}_{\ell})$ (\cite{Wintenberger02}, 1.2). Il suffit alors de trouver un entier $c$ qui annule le groupe fondamental de $\mathrm{Ker} \, \pi_{\ell}(\overline{\mathbb{Q}}_{\ell})$, qui est un sous-groupe semi-simple de $\mathrm{GL}_{2g}(\mathbb{C})$. Comme ses facteurs simples sont de rang borné par $2g-1$, ils ont au plus $2g-1$ poids minuscules et leur groupe fondamental est annulé par un entier inférieur ou égal à $2g$. On peut donc prendre pour $c$ le plus petit commun multiple des entiers $\leq 2g$, qui vérifie la condition annoncée. 
\end{proof}

\noindent
{\it Remarques.} On remarque que la constante $c$ donnée par la preuve vérifie : $$\log(c) = \sum_{p \in \mathcal{P}} \Big\lfloor\dfrac{\log(2g)}{\log p}\Big\rfloor
\log p \leq \pi(2g) \log(2g) \leq 3g,$$ où $\pi(2g)$, le nombre de premiers inférieurs ou égaux à $2g$, est majoré suivant \cite{Rosser-Schoenfeld62}, (3.7). On a donc : $c \leq e^{3g}.$

\medskip

 Par le corollaire 2.7.5 de \cite{Wintenberger02a}, si $g \leq 4$ ou si la conjecture de Mumford-Tate est vraie pour $A$, on peut prendre $c=1$ dans la proposition précédente. La conjecture de Lang est donc une conséquence de la conjecture de Mumford-Tate.

\medskip

En faisant la même hypothèse sur $\ell$, il est en fait possible de majorer le quotient $[\mathcal{H}_{\ell}(\mathbb{Z}_{\ell}): G_{\ell}]$, qui est fini par le théorème de Bogomolov. Par \cite{Zywina19}, théorème 1.2 (b), on peut trouver une borne qui dépend uniquement de $g$.

\section{Cohomologie galoisienne sur les corps locaux}

On étudie dans cette partie des groupes de cohomologie galoisienne introduits par Tate dans \cite{Tate67}, \S 3, dont on précise certains résultats en estimant les sauts de ramification d'une extension procyclique. L'étude de ces groupes de cohomologie nous servira dans la partie suivante pour transcrire la structure de Hodge-Tate d'une variété abélienne en une propriété d'algébricité de sa représentation abélienne $\ell$-adique.

\medskip

À l'exception du dernier paragraphe, on se place dans un cadre un peu plus général qui permet d'englober le cas des corps de fonctions en caractéristique positive. On fixe un nombre premier $\ell$ et un corps local $K_v$, complet pour une valuation $v$, dont le corps résiduel est parfait de caractéristique $\ell$, et dont l'indice de ramification est noté $e$. Soit $\bar{K_v}$ la clôture algébrique de $K_v$. La complétion de $\bar{K_v}$ pour la valeur absolue associée à $v$ est notée $C$. C'est un corps algébriquement clos.

\subsection{Préliminaires sur une extension procyclique totalement ramifiée} On fixe désormais une extension galoisienne $K_{\infty}$ de $K_v$ totalement ramifiée telle que : $$\mathscr{C}:= \mathrm{Gal}(K_{\infty}/K_v) \simeq \mathbb{Z}_{\ell},$$ et on se donne un générateur $\sigma$ de $\mathscr{C}$. On peut obtenir une telle extension en considérant un sous-corps bien choisi du corps engendré sur $K_v$ par les racines $\ell^n$-ièmes de l'unité, où $n$ varie dans $\mathbb{N}$. 

Si $n$ est un entier naturel, on note $K_n$ le sous-corps de $K_{\infty}$ correspondant au sous-groupe $\mathscr{C}(n):= \ell^n \mathbb{Z}_{\ell}$ de $\mathscr{C}$. On a : $$G(n):= \mathrm{Gal}(K_n/K_v) \simeq \mathscr{C}/\mathscr{C}(n).$$ Pour $\nu \geq -1$, le groupe de ramification en notation supérieure $G(n)^{\nu}$ est donné par (voir \cite{Serre62}, IV, \S 3, proposition 14, et la remarque qui suit la preuve de cette proposition pour passer  aux extensions profinies) : $$G(n)^{\nu}=\mathscr{C}^{\nu} \mathscr{C}(n)/ \mathscr{C}(n).$$ Soit $(\nu_i)_{i \geq -1}$ la suite définie par : $\mathscr{C}^{\nu}= \mathscr{C}(i)$, pour $\nu_i < \nu \leq \nu_{i+1}$. On a $\nu_0= -1$ et les $(\nu_i)_{i \geq 0}$ sont des entiers positifs par le théorème de Hasse-Arf (voir \cite{Serre62}, V, \S 7, théorème 1, pour le cas des extensions finies auquel on se ramène immédiatement). Pour une extension procyclique, on peut en dire un peu plus sur ces \og sauts de ramification \fg. 

\begin{lemma}
Il existe un entier $\kappa $ tel que pour tout $i \geq \kappa +1$ : $$\nu_{i+1}-\nu_i=e,$$ et $$\max \{ \nu_{\kappa +1}, \ell^{\kappa} + 1 \} \leq \frac{e \ell}{\ell-1}.$$  
\end{lemma}
\begin{proof}
Soit $\kappa$ le plus grand entier $i$ pour lequel : $$\nu_{i} < \frac{e}{\ell-1}.$$ Ce nombre est bien défini, car les $(\nu_i)_{i \geq 0}$ sont une suite strictement croissante d'entiers et $\nu_0=-1$. On applique maintenant \cite{Marshall71}, exemple page 280 (qui précise le théorème 6) à $K_n/K_v$ en prenant $n > \kappa +1$. On voit d'abord que pour $\kappa +1 \leq i \leq n$ : $$\nu_{i+1}-\nu_i=e \, ;$$ puis que : 
\begin{center}
$\nu_{\kappa+1} \leq \frac{e \ell}{\ell-1}$ \; ou \; $\nu_{\kappa+1} = \min \{ \ell \nu_{\kappa}, \nu_{\kappa} +e \} = \ell \nu_{\kappa}$, 
\end{center}
ce qui donne dans chaque cas la majoration annoncée pour $\nu_{\kappa + 1}$. De plus, pour tout $1 \leq i \leq n$, $$\nu_{i+1} \geq \min \{ \ell \nu_i, \nu_i + e \}.$$ Pour $i < \kappa$, ceci impose que $\nu_{i+1} \geq \ell \nu_i$, donc $$\ell^{\kappa-1} \leq \nu_{\kappa} < \frac{e}{\ell-1}.$$ Ceci valant pour tout $n > \kappa +1$, le lemme est entièrement démontré.
\end{proof}

\noindent
{\it Remarque.} Les travaux de Maus (\cite{Maus71}) et Miki (\cite{Miki81}) montrent que les conditions utilisées ici caractérisent la suite des sauts de ramification d'une extension cyclique de corps locaux, et que les estimations du lemme sont essentiellement optimales. 

\medskip

La différente relative d'une extension $M/L$ sera notée $\mathfrak{d}_{M/L}$. On peut estimer précisément les valuations de la différente sur la tour d'extensions $(K_n)_{n \in \mathbb{N}}$.
\begin{lemma}
On a : $$v(\mathfrak{d}_{K_n/K_v})=en + c_1 + \ell^{-n}c_2,$$ où $|c_1| \leq 6e^2$ et $|c_2| \leq 2\left(\frac{\ell e}{\ell-1}\right)^2-1 \leq 8e^2$. 
\label{sauts}
\end{lemma}
\begin{proof}
Suivant la proposition 5 de \cite{Tate67}, on voit que $|G(n)^{\nu}|=\ell^{n-i}$ si $\nu_i < \nu \leq \nu_{i+1}$ ($0 \leq i \leq n$). De plus :
\begin{eqnarray*}
v(\mathfrak{d}_{K_n/K_v}) & = & \int_{-1}^{\infty} \left(1- \frac{1}{|G(n)^{\nu}|} \right) d\nu \\
 & = & \sum_{i=0}^n(1-\ell^{i-n})(\nu_{i+1}-\nu_i) \\
 & = & \sum_{i=0}^{\kappa}(1-\ell^{i-n})(\nu_{i+1}-\nu_i) + e\sum_{i=\kappa +1}^n (1-\ell^{i-n}) \\
 & = & en - e \kappa + \nu_{\kappa +1}+1 - e \frac{\ell-\ell^{\kappa+1-n}}{\ell-1} - \sum_{i=0}^{\kappa}\ell^{i-n}(\nu_{i+1}-\nu_i) \\
 & = & en + c_1 + \ell^{-n}c_2, 
 \end{eqnarray*}
où $$c_1=-e\kappa + \nu_{\kappa +1} + 1 - \frac{e\ell}{\ell-1},$$ et $$c_2=\frac{e\ell^{\kappa+1}}{\ell-1} - \sum_{i=0}^{\kappa} \ell^i (\nu_{i+1}-\nu_i).$$
On a : $$|c_1| \leq 1 + \frac{2e\ell}{\ell-1} + e^2 \leq 6e^2,$$ et $$|c_2| \leq \left(\frac{e\ell}{\ell-1}\right)^2+ \ell^{\kappa}(1+\nu_{\kappa +1}) \leq 2\left(\frac{\ell}{\ell-1}\right)^2e^2-1 \leq 8e^2,$$
ce qui conclut la preuve.
\end{proof}

\medskip

\begin{corollary}
On a : $$v(\mathfrak{d}_{K_{n+1}/K_n})=e + \ell^{-n}c_3,$$
où $|c_3| \leq \frac{2e^2 \ell}{\ell-1} - 1+\frac{1}{\ell}  \leq 4e^2$.
\end{corollary}

\begin{proof}
Par transitivité de la différente (\cite{Serre62}, III \S 4, proposition 8), on obtient : $$v(\mathfrak{d}_{K_{n+1}/K_n})=e + \ell^{-n}c_2\frac{1-\ell}{\ell}.$$
Posons : $$c_3:=c_2\frac{1-\ell}{\ell}.$$ La majoration de $c_2$ dans le Lemme \ref{sauts} donne : $$|c_3| \leq 2e^2 \frac{\ell}{\ell-1} -\frac{\ell-1}{\ell}  \leq 4e^2,$$
et le corollaire est prouvé.
\end{proof}

\medskip

On peut maintenant évaluer la variation de la valeur absolue $|\cdot|$ associée à $v$ par application de la trace de $K_{n+1}$ à $K_n$, notée Tr$_{K_{n+1}/K_n}$.
\begin{corollary} 
Pour $x \in K_{n+1}$, on a : $$| \mathrm{Tr}_{K_{n+1}/K_n}(x) | \leq \ell^{-e+\ell^{-n}c_4} |x|,$$
où $|c_4| \leq   2e^2 \frac{\ell}{\ell-1} +\frac{1}{\ell} \leq 5e^2$.
\end{corollary}

\begin{proof}
Soit $\mathfrak{m}_n$ l'idéal maximal de l'anneau des entiers $R_n$ de $K_n$. Par $K_v$-linéarité de la trace, on peut supposer que $x \in R_{n+1}$ quitte à le multiplier par une puissance de $\ell$ suffisamment grande. On pose $\mathfrak{d}_{K_{n+1}/K_n}=  \mathfrak{m}_{n+1}^d$.  Par \cite{Serre62}, V \S 3, lemme 4 : $$\mathrm{Tr}_{K_{n+1}/K_n}(\mathfrak{m}_{n+1}^i)=\mathfrak{m}_n^j,$$ avec $j=\big[\frac{i+d}{\ell}\big]$. Comme $K_n$ est totalement ramifiée sur $K$ de degré $\ell^n$, on en déduit : 
\begin{eqnarray*}
v( \mathrm{Tr}_{K_{n+1}/K_n}(x)) & = & \ell^{-n}\big[\ell^n(v(x)+v(\mathfrak{d}_{K_{n+1}/K_n}))\big] \\
& = & e + \ell^{-n} [\ell^n v(x) + c_3] \\
& > & e + v(x) + \ell^{-n}(c_3-1).
\end{eqnarray*}
Le corollaire s'en déduit en passant à la valeur absolue et en posant $c_4=c_3-1$.
\end{proof}

\noindent
{\it Remarque.} Le lemme 3 de \cite{Serre62}, V \S 3 donne une expression assez simple pour $d$, mais qui nécessite de connaître le saut de ramification dans $K_{n+1}/K_n$.

\medskip

On définit une fonction $K_v$-linéaire $t$ sur $K_{\infty}$ par : $$t(x):= \ell^{-n} \mathrm{Tr}_{K_n/K_v}(x),$$
si $x \in K_n$. Cette formule ne dépend pas du choix de $K_n$ contenant $x$, par les propriétés élémentaires de la trace et car $[K_n/K_v]=\ell^n$. Rappelons qu'on a aussi fixé un générateur $\sigma$ du groupe $\mathscr{C}$. Si $x \in K_{n+1}$, par \cite{Tate67}, lemme 2 (et la remarque à la fin de la preuve du lemme): 
\begin{eqnarray}
\label{trace}
\big| x-\ell^{-1} \mathrm{Tr}_{K_{n+1}/K_n}(x) \big| \leq \ell^e \big|\sigma^{\ell^n}(x)-x \big| =  \ell^e \big|\sigma(x)-x \big|,
\end{eqnarray}
la dernière inégalité résultant de la formule du binôme (quitte à supposer $x$ entier par linéarité).
\medskip

Ceci va nous permettre de démontrer l'inégalité suivante.
\begin{proposition}
\label{operateur}
Pour tout $x \in K_{\infty}$ : $$|x-t(x)| \leq \ell^{c_5} |x- \sigma(x)|,$$ où $|c_5| \leq \frac{2\ell}{(\ell-1)^2}e^2+e+ \frac{1}{\ell(\ell-1)} \leq 6e^2$.
\end{proposition}
\begin{proof}
Prouvons par récurrence sur $n \geq 1$ que pour $x \in K_n$: $$|x-t(x)| \leq \ell^{u_n} |x- \sigma(x)|,$$
avec $u_1=e$ et si $n \geq 2$ : $$u_n= e + |c_4| \sum_{k=1}^{n-1} \ell^{-k}.$$ Pour $n=1$, c'est une conséquence immédiate de (\ref{trace}). Supposons le résultat montré au rang $n \geq 1$ et prenons $x \in K_{n+1}$. On applique l'hypothèse de récurrence à $y=$ Tr$_{K_{n+1}/K_n}(x)$ : 
\begin{eqnarray*}
|y-t(y)| & \leq & \ell^{u_n} |y- \sigma(y)| \\
 & \leq & \ell^{u_n} |\mathrm{Tr}_{K_{n+1}/K_n}(x-\sigma(x) | \\
 & \leq & \ell^{u_n - e + \ell^{-n}c_4} |x-\sigma(x)|.
\end{eqnarray*}
Par les propriétés de la trace et en appliquant (\ref{trace}), on en déduit que :
\begin{eqnarray*}
|x-t(x)| & = & |x-\ell^{-1}y + \ell^{-1}y -\ell^{-1}t(y)| \\
 & \leq & \max \{ |x-\ell^{-1}y|, |\ell^{-1}y -\ell^{-1}t(y)| \} \\
 & \leq & \max \{\ell^e, \ell^{u_n + \ell^{-n}c_4} \} |x-\sigma(x)| \\
 & \leq & \ell^{u_{n+1}} |x-\sigma(x)|.
 \end{eqnarray*}
 Comme pour tout $n$ : $$u_n \leq e + \frac{|c_4|}{(\ell-1} \leq \frac{2\ell}{(\ell-1)^2}e^2+e+ \frac{1}{\ell(\ell-1)} \leq 6e^2,$$ la proposition suit.
\end{proof}

\noindent
{\it Remarque.} Si on remplace le corps de base $K_v$ par $K_n$, où $n \geq 1$, la preuve peut être adaptée {\it mutatis mutandis} et on obtient exactement la même inégalité.

\subsection{Condition d'annulation de groupes de cohomologie}

\label{cohomologie}

Si $A$ est un $K_v$-espace vectoriel normé muni d'une action continue de $\mathscr{C}$, on peut définir les groupes de cohomologie $H^i(\mathscr{C}, A)$, où $i \geq 0$, pour les cochaînes continues. Le groupe $H^0(\mathscr{C}, A)$ est constitué des éléments de $A$ fixés par $\mathscr{C}$ ; le groupe $H^1(\mathscr{C}, A)$ est le quotient du groupe des {\it cocyles} continus de $\mathscr{C}$ dans $A$, c'est-à-dire des $\chi$ tels que : $$\forall \tau, \tau' \in \mathscr{C}: \; \; \chi(\tau \circ \tau')= \tau \cdot \chi(\tau') + \chi(\tau),$$ par son sous-groupe des {\it cobords}, c'est-à-dire des $\chi'$ tels qu'il existe $a \in A$ vérifiant  : $$\forall \tau \in \mathscr{C}: \; \; \chi'(\tau)= \tau \cdot a - a.$$
On renvoie à \cite{Serre62}, VII~\S3 pour plus de détails. 

\medskip

On considère maintenant la complétion $X$ de $K_{\infty}$ pour la valeur absolue associée à $v$. C'est un $K_v$-espace de Banach sur lequel $\mathscr{C}$ agit continûment. 
Dans ce cas, la structure des deux premiers groupes de cohomologie est clarifiée par les travaux de Tate (\cite{Tate67}, proposition 8).
\begin{lemma}
\label{cohomologie1}
On a $H^0(\mathscr{C}, X)= K_v$, et $H^1(\mathscr{C}, X)$ est un $K_v$-espace vectoriel de dimension $1$.
\end{lemma}

Soit $\chi$ un caractère continu de $\mathscr{C}$ à valeurs dans $K_v^*$. On note $X(\chi)$ le $\mathscr{C}$-module $X$ pour l'action \og tordue \fg : 
\begin{eqnarray*}\label{action-tordue}
\forall x \in X, \tau \in \mathscr{C} : \; \; \tau\cdot x:= \chi(\tau)\tau(x).
\end{eqnarray*}
Lorsque l'image du caractère $\chi$ n'est pas trop petite dans un sens qu'on va désormais préciser, l'action qu'il induit annule les groupes de cohomologie. Soit : $$c_6:=2+\bigg[ \frac{1}{e\ell(\ell-1)}+\frac{2\ell e}{(\ell-1)^2} + \frac{\log(2e)}{\log \ell}\bigg] \leq 7e.$$
\begin{proposition}
\label{cohomologie2}
Si $\chi^{\ell^{c_6}} \neq 1$, on a $H^0(\mathscr{C}, X(\chi))= 0$ et $H^1(\mathscr{C}, X(\chi))=0$.
\end{proposition}
\begin{proof}
On reprend les résultats de la proposition 7 de \cite{Tate67} en les précisant. La fonction $K_v$-linéaire $t$ définie sur $K_{\infty}$ se prolonge à $X$ par continuité. Si on note $X_0$ son noyau, l'espace $X$ est la somme directe de $X_0$ et de $K_v$. D'autre part, l'opérateur $\sigma - \mathrm{Id}$ admet un inverse $K_v$-linéaire et continu $\rho$ sur $X_0$ tel que : $$\forall y \in X_0: \; |\rho (y)| \leq \ell^{c_5} |y|.$$

\noindent
Soit $\lambda:= \chi(\sigma)^{-1}$. On vérifie immédiatement que le groupe $H^0(\mathscr{C}, X(\chi))$ est le noyau de $\sigma':= \sigma - \lambda \mathrm{Id}$ sur $X$, et que $H^1(\mathscr{C}, X(\chi))$ est un sous-groupe du conoyau de $\sigma'$. Il suffit donc de démontrer que $\sigma'$ est bijectif sur $X$. Comme $\lambda \neq 1$, la bijectivité de $\sigma'$ sur $K_v$ est immédiate et il suffit donc d'étudier $\sigma'$ sur $X_0$.

Par continuité de $\chi$, on observe que : $$|\lambda^{-\ell^n}-1| \longrightarrow 0,$$ donc $|\lambda^{-1}-1| <1$, et $|\lambda-1|<1$. Pour $i \geq 0$, soit $w_i:= v \big(\lambda^{\ell^i}-1\big)$. On a donc $w_0 \geq 1$. Par la formule du binôme et les propriétés usuelles des coefficients binomiaux, on obtient : $$\forall i \geq 0: \; w_{i+1} \geq \min \{ e + w_i , \ell w_i \},$$ et on a une égalité dès que $e+w_i \neq \ell w_i$. Soit $\kappa$ le plus petit entier $i \geq 0$ pour lequel $w_i > e/(\ell-1)$. Cet entier existe car la suite des entiers $(w_i)_i$ est strictement croissante. Remarquons que $\ell^{\kappa} \leq \ell e/(\ell-1)$, ce qui donne : $$\kappa \leq \bigg[ \frac{\log(2e)}{\log \ell }\bigg].$$ On en déduit que $c_6 > \kappa$. Il vient : 
\begin{eqnarray*}
w_{c_6} & \geq & \ell^{\kappa} +(c_6- \kappa)e \\
& \geq & \frac{e}{\ell-1}+(c_6- \kappa)e \\
& \geq & \frac{e}{\ell-1} + e +  \frac{1}{\ell(\ell-1)}+\frac{2\ell e^2}{(\ell-1)^2} \\
& > & |c_5|.
\end{eqnarray*}
Par hypothèse, $\lambda^{\ell^{c_6}} \neq 1$. Si $y \in X_0$, on a donc : $$\big|(\lambda^{\ell^{c_6}}-1) \rho(y)\big| \leq \ell^{c_5-w_{c_6}} |y| < |y|.$$ 
Si on prend $K_{c_6}$ comme corps de base au lieu de $K_v$, ces résultats restent vrais par la remarque suivant la preuve de la Proposition \ref{operateur}. Par linéarité de $\rho$, on obtient : $$\rho(\sigma^{\ell^{c_6}}-\lambda^{\ell^{c_6}})=1-(\lambda^{\ell^{c_6}}-1)\rho,$$
qui est donc inversible. On en déduit que $\sigma^{\ell^{c_6}}-\lambda^{\ell^{c_6}}$ est inversible sur $X_0$. Mais comme $\sigma^{\ell^{c_6}}-\lambda^{\ell^{c_6}}= \sigma' \tau$, pour un certain $\tau \in K_v[\sigma]$, l'application $\sigma'$ est également bijective.
\end{proof}

\noindent
{\it Remarques.} Si on suppose que $\ell \geq 4e^2$, on voit d'abord que $e/(\ell-1) <1$, donc $\kappa=0$. Mais on a également : $$|c_5| \leq e + \frac{1}{\ell^2-\ell} + \frac{1}{(\ell-1)^2}< 1+e.$$ Comme $w_1 \geq 1+e$, on peut alors choisir $c_6=1$.

\medskip

Si $K_v$ est de caractéristique $\ell$, on note que la borne obtenue dépend de l'indice de ramification, et pas des racines de l'unité contenues dans $K_v^*$ (liées à son corps des constantes), qui apparaissent naturellement dans l'étude de l'image de $\chi$. 

\subsection{Logarithme des caractères admissibles}

Les calculs qu'on vient d'effectuer sur une extension procyclique totalement ramifiée ont des conséquences importantes pour la théorie des modules de Hodge-Tate. On suppose dans ce paragraphe que $K_v$ est de caractéristique nulle. On note $\mathscr{G}:= \mathrm{Gal}(\bar{K_v}/K_v)$, $\mathscr{I}$ son sous-groupe d'inertie, et on fixe un caractère continu $\chi : \mathscr{G} \rightarrow K_v^*$. 
\begin{definition}
Le caractère $\chi$ est dit admissible (noté $\chi \sim 1$) s'il existe $x \in C^*$ tel que : $$\forall \sigma \in \mathscr{G} : \; \sigma(x)= \chi(\sigma)x.$$
\end{definition}

\noindent
On peut alors définir une relation d'équivalence entre deux caractères $\chi, \chi'$ : 
\begin{center}
$\chi \sim \chi'$ si et seulement si $\chi' \chi^{-1} \sim 1$. 
\end{center}
Soit $H^1(\mathscr{G}, C^*)$ le premier groupe de cohomologie de $\mathscr{G}$ à valeurs dans $C^*$ (les cochaînes étant supposées continues). Un caractère continu de $\mathscr{G}$ dans $C^*$ définit un élément de $H^1(\mathscr{G}, C^*)$ qui est nul si et seulement si le caractère est admissible. 

\medskip

Le caractère $\chi$ permet aussi de définir une action \og tordue \fg : $$\forall x \in C, \sigma \in \mathscr{G} : \sigma \cdot x = \chi(\sigma) \sigma(x).$$
On note $C(\chi)$ le $\mathscr{G}$-module ainsi obtenu. Si $\chi$ est admissible, le $\mathscr{G}$-module $C(\chi)$ est isomorphe à $C$, l'isomorphisme étant la multiplication par l'élément de $C^*$ intervenant dans la définition d'un caractère admissible. 

\medskip

On se donne un sous-corps $E$ de $K_v^*$ contenant $\mathbb{Q}_{\ell}$ et tel que $[E: \mathbb{Q}_{\ell}] < \infty$. On suppose que tous les conjugués de $E$ sur $\mathbb{Q}_{\ell}$ sont dans $K_v$, et on note $\Gamma_E$ l'ensemble des $\mathbb{Q}_{\ell}$-plongements de $E$ dans $K_v$. Le logarithme $\ell$-adique sur le groupe $U_E$ des unités de $E$, noté $\log$ (voir \cite{Serre68}, III, A2), permet d'interpréter certaines conditions d'admissibilité apparaissant naturellement en théorie de Hodge-Tate.
\begin{proposition}
\label{logarithme}
On suppose que $\chi$ est à valeurs dans $E^*$, et que pour $\tau \in \Gamma_E$, on a $\tau \circ \chi \sim 1$. Alors : $$\log \chi (\mathscr{I}) =0.$$
\end{proposition}

\begin{proof}
On reprend les arguments donnés dans la preuve de \cite{Serre68}, III, A3, proposition 3. Le sous-groupe $\log \chi (\mathscr{I})$ de $E$ est compact, isomorphe à $\mathbb{Z}_{\ell}^n$ pour un certain entier $n$. Supposons par l'absurde que $n \geq 1$.

En considérant un projecteur bien choisi, on voit qu'il existe une application $\mathbb{Q}_{\ell}$-linéaire : $f : E \rightarrow K_v$ telle que $f \big(\log \chi (\mathscr{I}) \otimes_{\mathbb{Z}_{\ell}} \mathbb{Q}_{\ell}\big)$ soit de dimension $1$. Par indépendance des caractères, l'ensemble $\Gamma_E$ est une base de $\mathrm{Hom}_{\mathbb{Q}_{\ell}}(E,K_v)$, ce qui donne une décomposition :  $$f= \sum_{\tau \in \Gamma_E} k_{\tau} \tau,$$ où les $k_{\tau}$ sont dans $K_v$. Le logarithme étant défini sur $\mathbb{Q}_{\ell}$, on a : $$f \circ \log \chi = \sum_{\tau \in \Gamma_E} k_{\tau} \log \tau \circ \chi.$$ Les hypothèses et la proposition 3 de \cite{Serre68}, III, A2 impliquent que l'image de  $f \circ \log \chi$ dans $H^1(\mathscr{G}, C)$ est nulle. Par les propriétés du logarithme (\cite{Serre86}, III, A2), quitte à multiplier $f$ par une puissance convenable de $\ell$, il existe un morphisme continu $$F: U_E \rightarrow U_{K_v}$$ tel que $f \circ \log = \log \circ F$. En utilisant une nouvelle fois la proposition 3 de \cite{Serre68}, III, A2, on voit que $F \circ \chi \sim 1$. De plus, $\log F \circ \chi (\mathscr{I}) \simeq \mathbb{Z}_{\ell}$, et $F \circ \chi (\mathscr{I}) \simeq \mathbb{Z}_{\ell} \times H$, pour un certain groupe $H$ fini. 

Soit $\mathscr{C}$ un quotient de $\mathscr{I}$ dont l'image est $\mathbb{Z}_{\ell}$. Il définit une extension $K_{\infty}/K_v$ totalement ramifiée de groupe de Galois $\mathscr{C}$ à laquelle on peut appliquer les résultats obtenus dans cette partie. Par le Lemme \ref{cohomologie1} et la Proposition \ref{cohomologie2} : $$H^1(\mathscr{C}, X) \neq H^1(\mathscr{C}, X(F \circ \chi)).$$ La proposition 10 de \cite{Tate67} donne alors : $$H^1(\mathscr{G}, C) \neq H^1(\mathscr{G}, C(F \circ \chi)),$$
ce qui contredit l'admissibilité de $F \circ \chi$.
\end{proof}

\section{Le cas des petits premiers : autour du théorème de Bogomolov}

La question des homothéties étant à peu près clarifiée pour les grands premiers, il s'agit désormais de trouver un exposant indépendant de $\ell$ dans le Théorème \ref{homothéties serre}. Dans la mesure où on dispose d'une borne pour l'entier $\ell_0$ qui apparaît dans le Théorème \ref{wintenberger}, on peut se contenter, pour les premiers $\ell \leq \ell_0$, d'un exposant dépendant de $A$ mais aussi de $\ell$. La question est donc de donner une version explicite du Théorème \ref{bogomolov}, dont on va reprendre la preuve en détail.

\subsection{Algèbre de Lie du groupe dérivé}

La première étape consiste à contrôler le groupe dérivé $H_{\ell}(\mathbb{Q}_{\ell})'$, pour pouvoir ensuite utiliser les propriétés des représentations abéliennes. Pour y arriver, on s'appuie sur un résultat classique de la théorie des groupes algébriques en caractéristique nulle.

\begin{proposition}
Pour $\ell \in \mathcal{P}$, le groupe $G_{\ell} \cap \mathcal{S}_{\ell}(\mathbb{Z}_{\ell})$ est d'indice fini dans $\mathcal{S}_{\ell}(\mathbb{Z}_{\ell})$.
\end{proposition}
\begin{proof}
Soit $\mathfrak{g}_{\ell}$ l'algèbre de Lie de $G_{\ell}$,  $\mathfrak{h}_{\ell}$ celle de $H_{\ell}(\mathbb{Q}_{\ell})$ et $\mathfrak{h}_{\ell}'$ celle de $H_{\ell}(\mathbb{Q}_{\ell})'$. Par la proposition 7.8 et le corollaire 7.9 de \cite{Borel69}, II, on a : $$\mathfrak{h}_{\ell}' = [\mathfrak{h}_{\ell}, \mathfrak{h}_{\ell}]=[\mathfrak{g}_{\ell}, \mathfrak{g}_{\ell}] \subset \mathfrak{g}_{\ell}.$$
On en déduit que $G_{\ell} \cap H_{\ell}(\mathbb{Q}_{\ell})'=G_{\ell} \cap \mathcal{S}_{\ell}(\mathbb{Z}_{\ell})$ est ouvert pour la topologie $\ell$-adique dans $\mathcal{H}_{\ell}(\mathbb{Q}_{\ell})'=\mathcal{S}_{\ell}(\mathbb{Q}_{\ell})$, donc d'indice fini dans $\mathcal{S}_{\ell}(\mathbb{Z}_{\ell})$.
\end{proof}

On note $s(A, \ell):=\big[\mathcal{S}_{\ell}(\mathbb{Z}_{\ell}) : G_{\ell} \cap \mathcal{S}_{\ell}(\mathbb{Z}_{\ell})\big].$ Trouver une borne explicite pour $s(A, \ell)$ est un problème très difficile si l'on ne fait pas d'hypothèse sur $A$. La stratégie de Serre ne semble pas pouvoir s'adapter aisément pour majorer $s(A, \ell)$ lorsque $\ell$ est un premier quelconque. 

\medskip

Des résultats ont été obtenus en petite dimension, notamment basés sur l'étude de \og l'algèbre de Lie entière \fg \, associée à $G_{\ell}$ (voir par exemple \cite{Lombardo15}, qui donne une version explicite du théorème de l'image ouverte de Serre sur les courbes elliptiques).

\subsection{Algébricité locale et structure de Hodge-Tate}

On souhaite maintenant étudier, pour $\ell$ quelconque, la représentation abélienne : $$\rho_{\ell}^{\mathrm{ab}} : G_K \longrightarrow H_{\ell}(\mathbb{Q}_{\ell})/H_{\ell}(\mathbb{Q}_{\ell})'.$$
La première étape est de se ramener à une représentation linéaire.

\begin{lemma}
Il existe un entier $c(g)>0$ et une représentation $$\sigma : H_{\ell}(\mathbb{Q}_{\ell}) \longrightarrow \mathrm{GL}_{n}(\mathbb{Q}_{\ell}),$$ telle que $n \leq c(g)$ et $\Ker \sigma = H_{\ell}(\mathbb{Q}_{\ell})'$. De plus, la composée $\sigma \circ \rho_{\ell}$ est de Hodge-Tate.
\end{lemma}
\begin{proof}
Il existe un nombre fini de sous-groupes algébriques semi-simples de $\mathrm{GL}_{2g}$, à isomorphisme près et indépendamment de $\ell$. En conséquence d'un théorème de Chevalley (voir \cite{Deligne82}, proposition 3.1.c), si $S$ est un tel groupe défini sur $\mathbb{Q}_{\ell}$, il est caractérisé par ses invariants tensoriels, c'est-à-dire qu'il existe un entier $c(S)>0$ tel que $S$ est le sous-groupe de $\mathrm{GL}_{2g}$ fixant : $$\Big\{x \in \bigoplus_{m \leq c(S)} (\mathbb{Q}_{\ell}^{2g})^{\otimes m}, \forall s \in S : s\cdot x=x \Big\}.$$
On prend $S:= \mathcal{S}_{\ell}$ et on considère le sous-espace vectoriel $W$ de $$\bigoplus_{m \leq c(S)} (\mathbb{Q}_{\ell}^{2g})^{\otimes m}$$ fixé par $S$. On a : $$\dim(W) \leq \sum_{m=0}^{c(S)} (2g)^m \leq (2g)^{c(S)+1} \leq (2g)^{\max_S c(S)+1}:=c(g).$$ La représentation $\sigma$ donnée par l'action de $H_{\ell}(\mathbb{Q}_{\ell})$ sur $W$ est de degré borné par $c(g)$, et on a par construction : $$\Ker \sigma = H_{\ell}(\mathbb{Q}_{\ell})'.$$ De plus, on sait que $\rho_{\ell}$ est de Hodge-Tate et que cette propriété est préservée par tensorisation (\cite{Bogomolov80}, 1, preuve du corollaire), donc la représentation $\sigma \circ \rho_{\ell}$ est encore de Hodge-Tate.
\end{proof} 

On pose $\sigma_{\ell}^{\mathrm{ab}}:=\sigma \circ \rho_{\ell}.$ L'image $T_{\ell}$ de  $H_{\ell}$ par $\sigma$ est un groupe algébrique commutatif qui n'a pas de composante additive, car celle-ci donnerait lieu à une extension abélienne non-ramifiée infinie de $K$. C'est donc un tore.

\begin{proposition}
On pose $m := n!(\ell^{n!}-1) \ell^{1+v_{\ell}(n!)}$. Alors : $$T_{\ell}(\mathbb{Z}_{\ell})^m \subset \sigma_{\ell}^{\mathrm{ab}} (G_K).$$
\end{proposition}

\begin{proof}
On commence par fixer une place $v | \ell$ et on considère la représentation $\sigma_v$ déduite de $\sigma_{\ell}^{\mathrm{ab}}$ par projection sur le groupe d'inertie $I_v$. D'après le théorème de Tate (\cite{Serre68}, III, A7), la représentation $\sigma_v$, qui est semi-simple puisque $T_{\ell}$ est un tore, est localement algébrique. Cela signifie qu'elle provient d'un morphisme algébrique du tore $T_v$ associé à $K_v$ sur $\mathbb{Q}_{\ell}$ via la théorie du corps de classe local -- modulo un sous-groupe fini de $I_v$. 

\medskip

On va quantifier ce fait. Pour cela, on se donne un facteur simple $V$ de $\mathbb{Q}_{\ell}^n$. Par le lemme de Schur, le commutant de $\sigma_v |_V$ est un corps $E$ vérifiant : $$[E: \mathbb{Q}_{\ell}] = \dim(V) \leq n,$$ et $\sigma_v |_V$ est donnée par un caractère $\chi_V$ à valeurs dans $E^*$. Soit $K'$ le compositum de $K_v$ et de la clôture galoisienne de $E$. La structure de Hodge-Tate de $\sigma_v$ permet de définir un morphisme algébrique $r_V$ défini sur le tore $T_{K'}$ associé à $K'$ et à valeurs dans $E^*$. Il est donné par la formule (\cite{Serre68}, III, A6, proposition 6) : $$r_V:= \prod_{\tau \in \Gamma_E} \tau^{-1} \circ \chi_{\tau E}^{n_{\tau}},$$ où $\Gamma_E$ est l'ensemble des $\mathbb{Q}_{\ell}$-morphismes $\tau : E \longrightarrow K'$, de poids $n_{\tau}$ dans la décomposition de Hodge-Tate, le caractère $\chi_{\tau E}: G_{K'}^{\mathrm{ab}} \longrightarrow \tau E^*$ se déduisant de l'inclusion $\tau E \subset K'$ par la théorie du corps de classe local sur $\tau E$ (\cite{Serre68}, III, A4). Le lien entre $r_V$ et $\chi_V$ est alors donné par les relations (\cite{Serre68}, III, A5, théorème 2) : $$\forall \tau \in \Gamma_E : \; \; \tau \circ ( r_V^{-1} \chi _V)  \sim 1.$$ 
Par la Proposition \ref{logarithme}, on en déduit : $$\log \Big( r_V^{-1} \chi_V \Big(I_v^{[K':K_v]}\Big) \Big)=0,$$
et $$r_V^{-1} \chi_V \Big(I_v^{[K':K_v]}\Big) \subset \Ker \log.$$ Par le lemme de \cite{Serre68}, III, A2, si $f_E$ est le degré résiduel de $E$ et si la composante $\ell$-primaire du sous-groupe de torsion de $E^*$ est de la forme $(\mathbb{Z}/ \ell \mathbb{Z})^{\alpha_E}$, on a : 
\begin{eqnarray*}
|\Ker \log | & = & (\ell^{f_E}-1)\ell^{\alpha_E} \\
& | &  (\ell^{[E:\mathbb{Q}_{\ell}]}-1)\ell^{1 +v_{\ell}([E: \mathbb{Q}_{\ell}])} \\
& | & (\ell^{n!}-1) \ell^{1+v_{\ell}(n!)}.\end{eqnarray*}
On a ainsi : $r_V^{m}= \chi_V^{m}$ sur $I_v$. L'entier $m$ ne dépendant pas du facteur simple $V$ considéré, il existe donc un morphisme algébrique $r_v$ défini sur le tore $T_v$ tel que  $r_v^{m}= \sigma_v^{m}$ sur $I_v$. 

\medskip

On peut finalement globaliser en considérant le produit des $T_v$, pour $v | \ell$ : il existe un tore $T$, image de $T_K$ par un certain morphisme algébrique $r$, tel que : $$T(\mathbb{Z}_{\ell})^m \subset \sigma_{\ell}^{\mathrm{ab}}(G_K) \subset T(\mathbb{Q}_{\ell}).$$ Mais comme $H_{\ell}$ est l'enveloppe algébrique de $G_{\ell}$, on a $T_{\ell}=T$ .
\end{proof}

\begin{corollary}
Soit $c := (\ell-1) \, \ell^{1+2v_{\ell}(c(g)!)}$. Alors : $$(\mathbb{Z}_{\ell}^{\times})^c \subset \rho_{\ell}^{\mathrm{ab}}(G_K).$$
\end{corollary}

\begin{proof}
Par la remarque de Deligne (\cite{Serre77}, 2, 3), on a : $$\mathbb{Z}_{\ell}^{\times} \subset H_{\ell}(\mathbb{Z}_{\ell}).$$ Puisque $\sigma$ induit une représentation fidèle de $H_{\ell}(\mathbb{Q}_{\ell})/H_{\ell}(\mathbb{Q}_{\ell})'$, la proposition précédente montre que l'indice $c$ de $\mathbb{Z}_{\ell}^{\times}$ dans $\rho_{\ell}^{\mathrm{ab}}(G_K)$ est fini et divise $m$. On en déduit que $c$ divise $(\ell-1)\ell^{v_{\ell}(m)}$, et comme : $$v_{\ell}(m) \leq 1+2v_{\ell}(c(g)!),$$ le corollaire est démontré.
\end{proof}

\noindent
{\it Remarques.} On observe que : $$c \leq \big( \ell c(g)!\big)^2.$$
Comme on l'a déjà vu, si $K$ est non ramifié au-dessus de $\ell$, on peut appliquer le corollaire 2.4 de \cite{Wintenberger02a} qui donne : $\mathbb{Z}_{\ell}^{\times} \subset \rho_{\ell}^{\mathrm{ab}} (G_K).$ Le résultat précédent ne sert donc que si $\ell$ divise $\Delta_K$.

\section{Homothéties et bornes uniformes pour la torsion}

On revient pour finir sur le problème de Manin-Mumford et la recherche de bornes uniformes, aussi explicites que possible, pour la torsion dans les sous-variétés de $A$. Les cas les plus favorables sont ceux où la constante de Serre peut être le mieux explicitée. Ils se produisent par exemple lorsque $A$ est une puissance de courbe elliptique ou une variété abélienne de type CM.

\subsection{Constante de Serre dans des cas particuliers}

Commençons par regarder le cas où $A=E$ est une courbe elliptique. Il est alors possible d'expliciter entièrement la constante de Serre. Une borne précise pour les grands premiers, reposant sur des travaux de Momose (\cite{Momose95}), est donnée par le théorème 39 de \cite{Eckstein05}. On note $h_K$ le nombre de classes de $K$.
\begin{theorem}[Eckstein]
On suppose que $E$ n'est pas de type CM. Si $$\ell \geq  \max \big\{(48 [K: \mathbb{Q}]h_K)^{3(48 [K: \mathbb{Q}]h_K)^2}, \Delta_{K(E[3])}\big\},$$ on a : $\mathbb{Z}_{\ell}^{\times} \subset G_{\ell}.$
\end{theorem}

Pour un résultat sans restriction sur $\ell$, on peut utiliser les estimations données par Lombardo dans le problème de l'image ouverte de Serre (\cite{Lombardo15}, théorème 1.1). On rappelle que la hauteur de Faltings de $E$ est notée $h_F(E)$.

\begin{theorem}[Lombardo]
\label{homotheties elliptiques}
On suppose que $E$ n'est pas de type CM. Alors $(\mathbb{Z}_{\ell}^{\times})^c \subset G_{\ell},$ où $$c \leq e^{1,9 \cdot 10^{10}}  \big( [K: \mathbb{Q}] \max \{ 1, h_F(E), \log [K : \mathbb{Q}] \}\big)^{12395}.$$
\end{theorem}

\medskip

\noindent
{\it Remarque.} Une étude fine de la représentation modulo $\ell$ est faite dans \cite{ADavid11}. Dans le cas des surfaces abéliennes, une borne pour les grands premiers peut être dérivée de \cite{Lombardo16}. 

\medskip

Les variétés abéliennes de type CM donnent lieu à des estimations particulièrement efficaces (voir \cite{Eckstein05}, théorème 6). Le problème porte alors essentiellement sur les représentations abéliennes et peut être attaqué grâce à la théorie du corps de classe dans le cas CM. 

\begin{theorem}[Eckstein]
\label{homotheties CM}
Si $A$ est de type CM, on a : $$\forall \ell \in \mathcal{P} \; : \; (\mathbb{Z}_{\ell}^{\times})^{c} \subset G_{\ell},$$ où $c=[K: \mathbb{Q}] \cdot |\mathrm{GL}_{2g}(\mathbb{F}_3)|.$
\end{theorem}

\noindent
La valeur de l'exposant est le plus souvent meilleure : si $\ell \nmid \Delta_K$, on peut prendre $c= |\mathrm{GL}_{2g}(\mathbb{F}_3)|$. Si, de plus, $A$ admet bonne réduction en $\ell$, on peut prendre $c=1$.

\subsection{Bornes uniformes dans le problème de Manin-Mumford} Commençons par rappeler certains des résultats obtenus dans \cite{Galateau-Martinez17}. On fixe un plongement de $A$ dans un espace projectif de dimension minimale $2g+1$ (suivant \cite{Shafarevich94}, 5.4, théorème 9). Si $V$ est une sous-variété de $A$, on note $\delta(V)$ le \og degré de définition \fg \, de $V$, c'est-à-dire le plus petit $d$ tel que $V$ est l'intersection d'hypersurfaces de $A$ de degré au plus $d$. 

\medskip

Par le théorème de Raynaud (\cite{Raynaud83a}), on sait que $V$ contient un nombre fini de translatés de sous-variétés abéliennes de $A$ par des points de torsion. Le nombre $T(V)$ de tels translatés qui sont maximaux pour l'inclusion peut être borné uniformément en fonction de $\delta(V)$, de $g$ et de la constante de Serre $c(A)$ (\cite{Galateau-Martinez17}, théorème 4.5 et la remarque qui suit la preuve). 
\begin{theorem}
On a :
$$T(V) \leq 16^{(c(A)+3)g^3} \delta(V)^g.$$
\end{theorem}

\noindent
Si $C$ est une courbe et $A$ est sa jacobienne de dimension $g \geq 2$, le cardinal des points de torsion de $C$, noté $|C_{\mathrm{tors}}|$, vérifie (\cite{Galateau-Martinez17}, proposition 3.7 et la remarque qui suit la preuve) : $$|C_{\mathrm{tors}}| \leq 4^{(2c(A)+2)g}.$$

\medskip

\noindent
{\it Remarque.} Il est probable qu'on obtienne de meilleurs résultats en séparant les petits premiers pour lesquels l'exposant d'homothétie est mal contrôlé, et les grands premiers pour lesquels il est borné uniformément. Pour ce faire, on pourrait introduire une isogénie dont le noyau contient le sous-groupe de torsion fini pour lequel la borne uniforme n'est pas réalisée. 

\medskip

\noindent
{\bf Cas des puissances de courbes elliptiques.} On suppose ici que $A=E^g$, où $E$ est une courbe elliptique non CM. Ce cas est particulièrement intéressant pour le problème de Manin-Mumford, dans la mesure où $A$ admet beaucoup de sous-variétés de torsion. Si $V$ est une sous-variété de $E^g$, le nombre de sous-variétés maximales est borné en utilisant le Théorème \ref{homotheties elliptiques} : $$T(V) \leq 16^{e^{2 \cdot 10^{10}}   g^3 \cdot \big( [K: \mathbb{Q}] \max \{ 1, h_F(E), \log [K : \mathbb{Q}] \}\big)^{12395}} \delta(V)^g,$$

\noindent
et si $C$ est une courbe de $E^g$ qui n'est pas la translatée d'une courbe elliptique par un point de torsion, on a : $$|C_{\mathrm{tors}}| \leq 4^{e^{2 \cdot 10^{10}}   g \cdot \big( [K: \mathbb{Q}] \max \{ 1, h_F(E), \log [K : \mathbb{Q}] \}\big)^{12395} }.$$

\medskip

\noindent
{\bf Cas des variétés abéliennes CM.} On suppose que $A$ est de type CM. Si $V$ est une sous-variété de $A$, le nombre de sous-variétés maximales est borné en utilisant le Théorème \ref{homotheties CM} : $$T(V) \leq 16^{ [K:\mathbb{Q}] \cdot 3^{5g^2}} \delta(V)^g.$$

\noindent
Si $C$ est une courbe de $A$ qui n'est pas la translatée d'une courbe elliptique par un point de torsion, on a : $$|C_{\mathrm{tors}}| \leq 4^{[K: \mathbb{Q}] \cdot 3^{5g^2}}.$$

\noindent
On pourra comparer avec les bornes classiques données par Coleman (\cite{Coleman85}) et Buium (\cite{Buium96}).

\subsection{L'effectivité dans le problème de Manin-Mumford} Le principe central pour résoudre effectivement un problème diophantien consiste à trouver des bornes calculables pour la hauteur de Weil et le degré de ses solutions. 

\medskip

Concernant le problème de Manin-Mumford, c'est théoriquement possible. Expliquons le cas d'une courbe $C$ plongée dans sa jacobienne $J$, de dimension $g$ et définie sur $K$. La hauteur de Weil d'un point de torsion $x \in J(\bar{K})$ est bornée de façon totalement explicite. En effet, sa hauteur canonique est nulle et la différence entre les deux hauteurs est contrôlée par les travaux de Manin et Zarhin (voir par exemple \cite{David-Philippon02}, proposition 3.9). On obtient : $$h(x) = \big|h(x) - \hat{h}(x) \big| \leq 4^{g+1}h_F(A)+3g \log(2).$$ 

D'autre part, on peut borner le nombre de points de torsion de $C$ en fonction de $g$, $h_F(J)$ et $\Delta_K$. Notons $N(J)$ le \og conducteur \fg \, de $J$, c'est-à-dire le produit des normes des idéaux premiers de mauvaise réduction de $J$. En combinant \cite{Rosser-Schoenfeld62}, (3.9) avec \cite{Pazuki16}, théorème 1.1 et le théorème de Minkowski, on peut trouver un premier $p \in \mathcal{P}$ qui ne divise pas $(2g+1) \cdot \Delta_K \cdot N(J)$ et vérifiant : $$p \leq (12g)^{25g^{12g^{4g}}}  \max \big\{ 1, \log \Delta_K, h_F(J) \big\}^2 .$$  Le théorème principal de \cite{Buium96} donne alors : $$|C_{\mathrm{tors}}| \leq (12g)^{26g^{12g^{4g}}}  \max \big\{ 1, \log \Delta_K, h_F(J) \big\}^{8g+4}.$$ On peut ensuite majorer le degré d'un point de torsion en remarquant que si $x \in C_{\mathrm{tors}}$, on a encore $x^{\sigma} \in C_{\mathrm{tors}}$ pour tout $\sigma \in G_K$. Il suit : $$[\mathbb{Q}(x): \mathbb{Q}] \leq [K:\mathbb{Q}] \cdot |C_{\mathrm{tors}}|.$$

\medskip

Les coordonnées projectives des points de torsion recherchés sont alors les solutions d'un nombre fini d'équations algébriques à coefficients entiers, dont les coefficients sont explicitement bornés. Si on dispose de bornes calculables pour le nombre de translatés de torsion maximaux d'une sous-variété de $A$, on peut par le même procédé déterminer explicitement ces translatés.

\bibliographystyle{smfplain}
\bibliography{biblio}

\providecommand{\bysame}{\leavevmode ---\ }
\providecommand{\og}{``}
\providecommand{\fg}{''}
\providecommand{\smfandname}{\&}
\providecommand{\smfedsname}{\'eds.}
\providecommand{\smfedname}{\'ed.}
\providecommand{\smfmastersthesisname}{M\'emoire}
\providecommand{\smfphdthesisname}{Th\`ese}
\begin{thebibliography}{10}

\bibitem{Beukers-Smyth02}
{\scshape F.~Beukers {\normalfont \smfandname} C.~Smyth} -- {\og Cyclotomic
  points on curves\fg}, \emph{Number Theory for the millenium, I} (2002),
  p.~67--85.

\bibitem{Bogomolov80}
{\scshape F.~Bogomolov} -- {\og Sur l'alg{\'e}bricit{\'e} des
  repr{\'e}sentations $l$-adiques\fg}, \emph{C. R. Acad. Sci. Paris}
  \textbf{290} (1980), no.~15, p.~701--703.

\bibitem{Borel69}
{\scshape A.~Borel} -- \emph{Linear algebraic groups}, Grad. Texts in Math.,
  Springer, 1969.

\bibitem{Buium96}
{\scshape A.~Buium} -- {\og Geometry of $p$-jets\fg}, \emph{Duke Math. J.}
  \textbf{82} (1996), p.~349--367.

\bibitem{Coleman85}
{\scshape R.~Coleman} -- {\og Torsion points on curves and $p$-adic abelian
  integrals\fg}, \emph{Ann. of Math.} \textbf{121} (1985), p.~111--168.

\bibitem{ADavid11}
{\scshape A.~David} -- {\og Borne uniforme pour les homoth{\'e}ties dans
  l?image de {G}alois associ{\'e}e aux courbes elliptiques\fg}, \emph{J. Number
  Theory} (2011), p.~2175--2191.

\bibitem{David-Philippon02}
{\scshape S.~David {\normalfont \smfandname} P.~Philippon} -- {\og Minoration
  des hauteurs normalis{\'e}es des sous-vari{\'e}t{\'e}s de vari{\'e}t{\'e}s
  ab{\'e}liennes {I}{I}\fg}, \emph{Comment. Math. Helv.} \textbf{77} (2002),
  p.~639--700.

\bibitem{Deligne82}
{\scshape P.~Deligne} -- \emph{Hodge cycles on abelian varieties}, Lecture
  Notes in Math., vol. 900, Springer, 1982.

\bibitem{Grothendieck70}
{\scshape M.~Demazure {\normalfont \smfandname} A.~Grothendieck} --
  \emph{Structure des schémas en groupes réductifs, {SGA 3, III}}, Lecture
  Notes in Math., vol. 153, Springer, 1970.

\bibitem{Eckstein05}
{\scshape C.~Eckstein} -- {\og Homoth{\'e}ties, {\`a} chercher dans l'action de
  {G}alois sur des points de torsion\fg}, \emph{Th{\`e}se de {D}octorat,
  {U}niversit{\'e} de {S}trasbourg} (2005).

\bibitem{Galateau-Martinez17}
{\scshape A.~Galateau {\normalfont \smfandname} C.~Martinez} -- {\og A bound
  for the torsion on subvarieties of abelian varieties\fg},
  \emph{Pr\'epublication} (2017).

\bibitem{Gaudron-Remond20}
{\scshape {\'E}.~Gaudron {\normalfont \smfandname} G.~R{\'e}mond} -- {\og
  Nouveaux th{\'e}or{\`e}mes d'isog{\'e}nie\fg}, \emph{Pr{\'e}publication}
  (2020).

\bibitem{Grothendieck72}
{\scshape A.~Grothendieck} -- \emph{Groupes de monodromie en géométrie
  algébrique, {SGA 7, I}}, Lecture Notes in Math., vol. 288, Springer, 1972.

\bibitem{Hindry88}
{\scshape M.~Hindry} -- {\og Autour d'une conjecture de {S}erge {L}ang\fg},
  \emph{Invent. Math.} \textbf{94} (1988), p.~575--603.

\bibitem{Lang65}
{\scshape S.~Lang} -- {\og Division points on curves\fg}, \emph{Ann. Mat. Pura
  Appl.} \textbf{70} (1965), no.~1.

\bibitem{Larsen-Pink97}
{\scshape M.~Larsen {\normalfont \smfandname} R.~Pink} -- {\og A connectedness
  criterion for $\ell$-adic galois representations\fg}, \emph{Israel J. Math.}
  \textbf{97} (1997), p.~1--10.

\bibitem{Lombardo15}
{\scshape D.~Lombardo} -- {\og Bounds for serre's open image theorem for
  elliptic curves over number fields\fg}, \emph{Algebra Number Theory}
  \textbf{9-10} (2015), p.~2347--2395.

\bibitem{Lombardo16}
\bysame , {\og Explicit surjectivity for galois representations attached to
  abelian surfaces\fg}, \emph{J. Algebra} \textbf{460} (2016), p.~26--59.

\bibitem{Marshall71}
{\scshape M.~Marshall} -- {\og Ramification groups of abelian local field
  extensions\fg}, \emph{Can. J. Math.} \textbf{23} (1971), no.~2.

\bibitem{Maus71}
{\scshape E.~Maus} -- {\og On the jumps in the series of ramifications
  groups\fg}, \emph{Mém. Soc. Math. Fr.} \textbf{25} (1971), p.~127--133.

\bibitem{Miki81}
{\scshape H.~Miki} -- {\og On the ramification numbers of cyclic $p$-extensions
  over local fields\fg}, \emph{J. Reine Angew. Math} \textbf{328} (1981),
  p.~99--115.

\bibitem{Momose95}
{\scshape F.~Momose} -- {\og Isogenies of prime degree over number fields\fg},
  \emph{Compos. Math} \textbf{97} (1995), no.~3, p.~329--348.

\bibitem{Pazuki16}
{\scshape F.~Pazuki} -- {\og Heights, ranks and regulators of abelian
  varieties\fg}, \emph{Ramanujan Math. Soc. Lecture Note Series} \textbf{26}
  (2016).

\bibitem{Raynaud83a}
{\scshape M.~Raynaud} -- {\og Sous-vari{\'e}t{\'e}s d'une vari{\'e}t{\'e}
  ab{\'e}lienne et points de torsion\fg}, \emph{Prog. Math.} \textbf{35}
  (1983), p.~327--352.

\bibitem{Remond18}
{\scshape G.~R{\'e}mond} -- {\og Conjectures uniformes sur les vari{\'e}t{\'e}s
  ab{\'e}liennes\fg}, \emph{Quart. J. Math.} \textbf{69} (2018), p.~459--486.

\bibitem{Rosser-Schoenfeld62}
{\scshape G.~Rosser {\normalfont \smfandname} L.~Schoenfeld} -- {\og
  Approximate formulas for some functions of prime numbers\fg}, \emph{Illinois
  J. Math.} \textbf{6} (1962), no.~1, p.~64--94.

\bibitem{Serre62}
{\scshape J.-P. Serre} -- \emph{Corps locaux}, Hermann, Paris, 1962.

\bibitem{Serre68}
\bysame , \emph{Abelian l-adic representations and elliptic curves}, Benjamin,
  New York, 1968.

\bibitem{Serre77}
\bysame , {\og Représentations $\ell$-adiques\fg}, \emph{Kyoto Symposium on
  Algebraic Number Theory} (1977), p.~177--193.

\bibitem{Serre86}
\bysame , \emph{{O}euvres, {C}ollected papers, {IV}}, Springer-Verlag, Berlin,
  1985-1998.

\bibitem{Serre13}
\bysame , {\og Un crit{\`e}re d'ind{\'e}pendance pour une famille de
  repr{\'e}sentations $\ell$-adiques\fg}, \emph{Comment. Math. Helv.}
  \textbf{88} (2013), p.~541--554.

\bibitem{Shafarevich94}
{\scshape I.~R. Shafarevich} -- \emph{Basic algebraic geometry. 1}, second
  \smfedname, Springer-Verlag, Berlin, 1994.

\bibitem{Tate67}
{\scshape J.~Tate} -- {\og $p$-{D}ivisible {G}roups\fg}, \emph{Proceedings of a
  Conference on Local Fields} (1967), p.~158--183.

\bibitem{Wintenberger02}
{\scshape J.~Wintenberger} -- {\og D{\'e}monstration d'une conjecture de {L}ang
  dans des cas particuliers\fg}, \emph{J. Reine Angew. Math.} \textbf{553}
  (2002), p.~1--16.

\bibitem{Wintenberger02a}
\bysame , {\og Une extension de la théorie de la multiplication complexe\fg},
  \emph{J. Reine Angew. Math.} \textbf{552} (2002), p.~1--14.

\bibitem{Zywina19}
{\scshape D.~Zywina} -- {\og An effective open image theorem for abelian
  varieties\fg}, \emph{Pr\'epublication} (2019).

\end{thebibliography}

\end{document}